\documentclass[12pt,english]{article}
\usepackage{amsthm}
\usepackage{amsmath}
\usepackage{amssymb}
\usepackage{esint}
\usepackage[authoryear]{natbib}

\makeatletter
\numberwithin{equation}{section}
\numberwithin{figure}{section}
  \theoremstyle{plain}
  \newtheorem*{claim*}{\protect\claimname}
\theoremstyle{plain}
\newtheorem{thm}{\protect\theoremname}
  \theoremstyle{plain}
  \newtheorem{cor}{\protect\corollaryname}
 \theoremstyle{definition}
 \newtheorem*{defn*}{\protect\definitionname}
  \theoremstyle{plain}
  \newtheorem{prop}{\protect\propositionname}
  \theoremstyle{remark}
  \newtheorem*{rem*}{\protect\remarkname}

\usepackage{fullpage}
\usepackage{mathrsfs} 
\usepackage{endfloat} 
\usepackage[colorlinks=true, linkcolor=cyan, citecolor=black]{hyperref}
\usepackage{titlesec}
\titleformat{\section}{\large\sc\center}{\thesection}{1em}{}

\allowdisplaybreaks[4]

\makeatother

\usepackage{babel}
  \providecommand{\claimname}{Claim}
  \providecommand{\definitionname}{Definition}
  \providecommand{\propositionname}{Proposition}
  \providecommand{\remarkname}{Remark}
\providecommand{\corollaryname}{Corollary}
\providecommand{\theoremname}{Theorem}

\begin{document}

\title{Looking Backward and Looking Forward%
\thanks{The authors would like to express their gratitude to Ken Judd and
Peter C. B. Phillips for useful discussions at the early stage of
this manuscript, and the participants in the seminars at University of
Amsterdam. All the remaining errors are ours.%
}}

\author{Zhengyuan Gao%
\thanks{
Center for Operations Research and Econometrics (CORE), Universite
catholique de Louvain, Voie du Roman Pays 34, B-1348, Louvain-la-Neuve,
Belgium. E-mail: \texttt{zhengyuan.gao@uclouvain.be} %
}\\
 \and Christian M. Hafner%
\thanks{Institut de Statistique, Biostatistique et Sciences Actuarielles (ISBA)
and Center for Operations Research and Econometrics (CORE), Universite
catholique de Louvain, Voie du Roman Pays 20, B-1348, Louvain-la-Neuve,
Belgium. Email: \texttt{christian.hafner@uclouvain.be}, corresponding author. %
}\\
 }
\maketitle
\begin{abstract}
Filtering has had a profound impact as a device of perceiving information and deriving agent expectations
in dynamic economic models.
For an abstract economic system, this paper shows that the foundation of applying the filtering method
corresponds to the existence of a conditional expectation as an equilibrium
process. Agent-based rational behavior of looking backward and looking
forward is generalized to a conditional expectation
process where the economic system is approximated by a
class of models, which can be represented and estimated
without information loss. The proposed framework elucidates
the range of applications of a general filtering device and is not limited
to a particular model class such as rational expectations.
\end{abstract}
\thispagestyle{empty}



\newpage{}

\baselineskip19.15pt \pagenumbering{arabic}

\section{Perception as a Filter\label{sec:Perception}}

Human aspirations and desires imply forward-looking decisions.
Making a forward-looking decision requires the construction of an expectation based on the information that is backward induced. It can be characterized by a process of making conditional expectations.
Such a process allows us to construct our subjective beliefs rather than perceiving the world as mere presentation.
Instead, we perceive the world as an object of perception in which our own experience and knowledge are integrated.
We become unified with that perception. Thus formulating an expectation of an agent cannot separate the perceiver from the perception. From the market's perspective, an agent no longer views himself as an individual, but rather becomes a ``cognitive subject'' of a time-invariant perception  where the laws of the economy are revealed. The process of forming conditional expectations is the practical consequence of this identity
as it attempts to represent our immersion with the world, and these attempts constitute the essential laws of the world.
We refer to this process as a {\em filter}.

The equilibrium in rational expectation models is based on the assumption
that the agents in the model are confident with their perceptions.
As a consequence, the agents trust that the optimal actions following
their expectations will give them maximum utilities or profits.
Sequential decisions are made under expectations conditioning
on past information. Such decision processes induce a representation%
\footnote{The representation is defined in \citet{KoopmansRubinLeipnik1950}
as ``a way of writing the system''. In general, the representation
is a way of presenting the law of motion of this system. %
} of the conditional probability distribution of the economic dynamics:
the law of motion perceived by agents will sequentially influence
the law of motion that agents actually face. These perceptions can be viewed as filters as they are
often characterized by recursive projection schemes \citep{Simon1959}.
There is, however, a dichotomy in the understanding of filters in economic theory and econometrics.

In economic theory, the filtering method (perception) is an \emph{active}
process involving the agent's attention to a small part of the whole
dynamic system and excluding almost everything out of the scope of
their attention. In econometrics, the filtering method is treated
as a \emph{passive} process selecting some statistical relevant information
of a given dynamical model. Both aspects of the filtering method are
mainly treated by \citet{Sargent1987} and \citet{Hamilton1994}, respectively. Some other, similar types of
perceptions have been discussed in \citet{MarcetSargent1989}, \citet{MarcetSargent1989JET}, \citet{HansenSargent2007}, and
\citet{HansenPolsonSargent2010}. Rather than considering a specific
economic or econometric model, this paper characterizes general
perceptions that are concealed in abstract models where both the active
and passive arguments can be integrated.

\cite{Hansen2007} examines the inference of rational expectation models from two separate perspectives. The way he reduces the gap between these two perspectives is to enlarge the models used by the economic agents. We extend the scope to a more general
class of models for the economic agents under which the integration becomes more natural. The filtering method provides formal representation-estimation
processes for practical situations. The relevance of these processes in our setting is that they are not merely statistical
techniques, but actual dynamic mechanisms used in expectations and perceptions.
The equations and assumptions that appear in the estimation
procedure correspond to the perception of economic agents. This is in the spirit of
 \citet{Klein1950}: ``The purpose in building econometric
models is to describe the way in which the system actually operates.[...] The construction of such a system is a task in which economic
theory and statistical method combine.''

In our context, building econometric models is related to the proper
specification of filters. Considering filters as a perception
device in an abstract economy induces a large class of
econometric models. The remaining econometric task
is to reduce the abstract representation to
a feasible form for estimation. On the other hand, as the complexity
of the environment increases, agents learn more and more about the
mechanisms and processes that are used to relate themselves to that
environment and to achieve their goals.
The availability of general implementable techniques in econometrics
will elucidate inaccessible places for the abstract models.
As the comment in \citet{Simon1959}
on modeling human expectation says, ``it is one thing to have a set of
differential equations, and another thing to have their solutions.''
Economic theory predominates in the definition of the representation
describing a certain type of economic dynamics, while econometric methods
are associated with the determination of the agent's way of estimation.

The need for reconciling economic theory and econometrics may not be obvious in linear or linearizable
structural models when the laws of motion are
specified on either theoretical or empirical grounds, and hence either side of the coin will be
sufficient to justify the model. However, if we start with an abstract
economy, the limitation of modeling tools in this complex environment
will force us to integrate all available factors. The generalization
of the filtering mechanism will give
a fundamental interpretation of agent's perceptions in complex situations.
All subsequent statistical procedures, such
as estimation, inference and forecasting, will more or less depend
on the way this generalization has been formulated.
Our contribution will be to make this generalization available.

Early attempts in this direction were made in the framework of rational
expectation models. The expected utility or profit for each agent
depends on the assumption of the agent's perception mechanism.
\citet{EvansRamey1992} show how agents adjust
their long-term expectations under different perception rules or predictions.
The standard method assumes that the agent's prediction
uses a single, presumably correct law of motion.
But if perceptions of agents differ, then the corresponding predictions are incompatible with
each other. A sequence of works by Hansen and Sargent, covered by the
monograph \citet{HansenSargent2007}, introduces the concern of robustness
of agent's expectations. The main idea is that agent's decisions contain their prior
worries about a possible mis-specification of the model. These multiple
priors generalize the perception or the law of motion contained in the
agent's mind. The associated perception mechanism for the robust decision
agents is also a filter, called {\em robust filter}, and is a dual of the
linear-quadratic regulator problem of utilities. Given the robustness
concerns, the agent makes the expectation based on a \emph{class}
of models whose information is presumably not far from the underlying
model in a certain metric. Therefore, the robust filter
generalizes the mechanism used in the single law case.

Our motivation is related to the robust filters of \citet{HansenSargent2007} and \citet{Hansen2007},
but the relation is more one of spirit rather than of a precise form.
In terms of the robustness framework, our objective of generalizing filtering mechanisms attempts
to analyze how large the class of alternative
models could be while remaining consistent with some general filter. The class used in
\citet{HansenSargent2007} is restricted by a risk threshold
and is equivalent to a class of partially specified processes. If
we enlarge such a class to an abstract economic system, is the filtering type
mechanism still an optimal choice for agents in this system?
We ask this question because rational expectation models are merely an approximation of
the real world. Exploring the applicable range of a general filtering device will demonstrate
its usefulness essentially irrespective of the particular form of the economic model.

The paper aims at making the previous assertions rigorous. The mathematical
tools we use are borrowed from stochastic analysis and stochastic control.
For expositional purposes we delay the formal results of the paper to Section 4.
Section 2 introduces the economic model as a general probability space and anticipates the main result on the existence
of filters. In Section 3 we define three claims under which the economic system is supposed to operate and which will
allow us to obtain an explicit representation of the filter.
Section 5 elaborates on the claims in asset pricing, on a framework with stochastic volatility, and on the link of the
general filtering results with those in a linear framework. Section 6 summarizes the main findings. All proofs are delegated to an appendix.

\section{The Model\label{sec:The-Model}}

\subsection{An Abstract Economy\label{sub:An-Abstract-Economy}}

The economic system in this paper is driven by components
whose evolution is modeled via stochastic processes. We stack
these components into a \emph{state vector} and denote it
as $X_{t}=\{X_{i,t},t>0,i=1,\dots, I_{X}\}$. When we state specifically
$t\in\mathbb{Z}^{+}$, $X_{t}$ is a discrete time process, otherwise
$X_{t}$ is assumed to be a continuous time process. Throughout the paper, $x$
refers to either a deterministic variable or a realization of $X_{t}$.
The state vector $X_{t}$ may consist of unobservable features such as private information,
utilities or underlying prices.

The \emph{underlying abstract economic model} in our context is a
probability space $(\Omega,\mathcal{F},\mathbb{P})$,  where we define
$X$ together with a filtration $(\mathcal{F}_{t})_{t\geq0}$. The
filtration $\mathcal{F}_{t}$ is right continuous, $\mathcal{F}_{t}=\cap_{\epsilon>0}\mathcal{F}_{t+\epsilon}$,
and $\mathcal{F}=\lim_{t\rightarrow\infty}\mathcal{F}_{t}$.
In $\mathcal{F}$, there is a $\mathbb{P}$-null set%
\footnote{A set $A$ is called $\mathbb{P}$-null set if $A$ is measurable
on $(\Omega,\mathcal{F},\mathbb{P})$ and $\mathbb{P}(A)=0$. %
} contained in $\mathcal{F}_{0}$ and consequently in all $\mathcal{F}_{t}$.

The values of the states of $X_{t}$ form a measurable space $(\mathbb{S},\mathcal{S})$.
The state space $\mathbb{S}$ is a compact metric space and is associated
with a Borel $\sigma$-algebra $\mathcal{S}=\mathcal{B}(\mathbb{S})$.
We assume $X_{t}$ to be measurable and the measurable mapping is:
\[
X_{t}(\omega):\left([0,\infty)\times\Omega,\,\mathcal{B}([0,\infty))\otimes\mathcal{F}\right)\rightarrow(\mathbb{S},\mathcal{S}),
\]
 where $\otimes$ denotes the product operator for $\sigma$-fields.

While the essential features of economic dynamics are assumed to be
captured by the state variables $X_{t}$, the observable economic
variables and public information, in general, are not. Since the observable and
public information is the major resource for agents to make their
expectations about how the economic states $X_{t}$ change, we specify
it by another process $Y_{t}$. Let $Y_{t}$ include those observable
variables that are related with $X_{t}$, and assume that the dimension of $Y_t$ is not larger
than that of $X_t$. The \emph{observable information
set} is $\mathcal{Y}:=\vee_{t\in\mathbb{R}^{+}}\mathcal{Y}_{t}$, the
filtration generated by the observable process $Y_{t}$ such that
\[
\mathcal{Y}_{t}:=\sigma(Y_{s},s\in[0,t])\vee\mathcal{N}
\]
with $t\geq0$, where $\mathcal{N}$ is the collection of all $\mathbb{P}$-null
sets of our economic model $(\Omega,\mathcal{F},\mathbb{P})$.\footnote{The notation
$A\vee B$ means that the set is generated by $A$ and $B$.}
The available
information $\mathcal{Y}_{t}$ is induced by observations up to time
$t$ and thus it will be used for making inference about $X_t$. Since
$\mathcal{F}_{t}$ is right continuous, to make $Y_{t}$ compatible
with $X_{t}$, we assume that the filtration $\mathcal{Y}_{t}$ is
also \emph{right continuous}.

Agents will construct
their perceptions of states $X_{t}$ based on the information in $\mathcal{Y}_{t}$.
It means that agents are able to construct $\pi_{t}$, the \emph{conditional
distribution} of $X_{t}$, given $\mathcal{Y}_{t}$. The conditional
expectations characterize the perceptions or filters of $X_{t}$.
For any $t$, the conditional distribution is a stochastic process
$(\omega,t)\mapsto\pi_{t}(\omega)$ such that
\[
\pi_{t}\left(X_{t}(\omega)\in A\right)=\mathbb{P}\left[X_{t}\in A|\mathcal{Y}_{t}\right](\omega),\qquad A\subset\mathcal{S}.
\]
For simplicity, we will write $\pi_{t}(\omega)$ as $\pi_{t}$ in short.

Since the perceptions are processes, any valuation of $X_{t}$ will
involve an expectation w.r.t. the conditional distribution process.
The definition of conditional expectation of $X_{t}$ is restricted
to an equivalence class of $\mathcal{Y}_{t}$-measurable $X$ such
that:
\[
\mathbb{P}[X_{t}\cap B]=\mathbb{P}[X\cap B],\qquad X,B\subset\mathcal{Y}_{t}.
\]
Then the expectation of a function $\varphi(X_{t})$ can be expressed
as $\int\varphi(x)\pi_{t}(dx)$. The conditional expectation of $\varphi$ is ultimately what is desired from filtering, but the methods for obtaining the conditional distribution process are quite involved. If this integral is well-defined for a class of functions $\varphi$, then we call
them \emph{choice functions} $\varphi$.

\subsection{Existence of Filters}

In the abstract economy $(\Omega,\mathcal{F},\mathbb{P})$, we consider
perception equivalently as a filter. But perception as a common human
behavior should always exist on either individual or aggregrate levels.
Will filters always exist in this abstract economy? Due to the $\mathbb{P}$-null
set $\mathcal{N}\subset\mathcal{Y}_{t}$, $\pi_{t}$ may in fact not
be well defined for all $\omega\in\Omega$ but only for $\omega$
outside the $\mathbb{P}$-null set. Thus, the question of existence
of $\pi_t$ is equivalent to the question under which circumstances
one can gain sufficient control over all $\mathbb{P}$-null sets $\mathcal{N}$
such that the expectation $\int\varphi(x)\pi_{t}(dx)$ is well-defined
for choice functions $\varphi$. In other words, the filters exist
in the economy when perceptions induce well-defined expectations.

The theorem of the existence of filters in our abstract economy will
be given in Section \ref{sec:Main-Results}, Theorem \ref{thm:ExistenceConditionalProb}.
Here, we discuss the consequences of this theorem without presenting
too many technical details. Suppose a process $^{o}\varphi(\cdot)$ can be thought of as the $\mathcal{Y}_{t}$-measurable
representation of the choice function $\varphi(\cdot)$. The theorem states that given some regularity
conditions, for any choice function $\varphi$, the expectation of
$\varphi(X_{t})$ w.r.t. the filter $\pi_{t}$ exists and is equal to
the process $^{o}\varphi(X_{t})$ under the $\mathbb{P}$ measure.
In other words, agent's perceptions of $\varphi(\cdot)$ coincide
with the observable information. Therefore, the existence of $^{o}\varphi(\cdot)$
induces the existence of $\pi_{t}(\cdot)$ for the choice function
$\varphi(\cdot)$ and vice versa.

Note that although relations between $X_{t}$ and $Y_{t}$
exist, expectations
conditional on $\mathcal{Y}_{t}$ do not necessarily coincide
with those conditional on $\mathcal{F}_{t}$. In particular, the $\mathbb{P}$-null
information originates in $\mathcal{F}_{t}$, but in $\mathcal{Y}_{t}$ this information contains unpredictable events that may happen.
Once agents observe these unpredictable events, their perceptions will be influenced.
We will emphasize this point in the following subsection.

\subsection{The Importance of the $\mathbb{P}$-null Set }

Although it complicates the set-up, the $\mathbb{P}$-null set is a crucial
feature in $(\Omega,\mathcal{F},\mathbb{P})$. Apart from its mathematical
characteristics, it is meaningful in economic problems and
affects our way of evaluating a model using empirical
data.

The role of the $\mathbb{P}$-null set in defining a conditional probability
has first been illustrated by Kolmogorov in his famous \emph{Borel\textendash Kolmogorov
paradox}. The paradox shows that the conditional probability is not
uniquely defined with respect to a null set, see \citet[Chapter 5]{Kolmogorov1956} and \citet[Chapter 2]{BainCrisan2008}.
From an economic perspective, one can think of the $\mathbb{P}$-null
set on $\mathcal{F}$ and $\mathcal{Y}$ as those \emph{unexpected}
\emph{events} which have been included in the underlying economic
mechanism $\mathcal{F}$ and in the agent's observable information
set $\mathcal{Y}$.

Although most events in the $\mathbb{P}$-null set of $\mathcal{F}$
correspond to events in the $\mathbb{P}$-null set of $\mathcal{Y}$, the two null sets
are not equivalent. To see the subtle difference, let us assume that the $\mathbb{P}$-null
events in $\mathcal{Y}$ result from aggregating countable
$\mathbb{P}$-null events in $\mathcal{F}$. The aggregation leads to uncountable events which
are too ``complex'' to be embedded in the underlying model, the probability
space $(\Omega,\mathcal{F},\mathbb{P})$. The model $(\Omega,\mathcal{F},\mathbb{P})$
attributes zero-measures for any countable event sets that are beyond
its explanatory power, but for uncountable event sets the model cannot even
affirm their existence.

We illustrate the economic meaning of some $\mathbb{P}$-null sets
on $\mathcal{Y}$ by a concept which we call \emph{overflow}. The
effect of this overflow is related to the regularization of the $\mathbb{P}$-null
set on $\mathcal{F}$, which is a result of Theorem \ref{thm:ExistenceConditionalProb}\emph{.}

To give an example of overflow, consider {\em economic bubbles}.
There is a long debate whether or not economic
bubbles exist. Rather than joining the debate, our intention here
is to use bubbles as an example to illustrate overflow characteristics.
Suppose some individual gamblers have complex trading strategies, and their gains are publicly observable.
These speculative trades, therefore, are included in
the information set $\mathcal{Y}_{t}$ at the agents' disposal.
However, the strategies behind these trades may not be fathomable
by the public and are conducted in manyfold
ways, such as forbidden disclosures (private information), special
technical equipments (e.g. high-frequency trading), or even improper policies (lobbying).
Any economic model that wants to cover some or all of these specific
features will make its complexity explode.
This limitation is recognized
by the public, and hence it is reasonable for the public
to believe that the underlying economic model $(\Omega,\mathcal{F},\mathbb{P})$
will set zero measure on each of these strategies and the associated
actions because they are unexplained by the model. In other
words, each action of the trading strategy is in the $\mathbb{P}$-null
set on $(\Omega,\mathcal{F},\mathbb{P})$.

The economic bubble can
be considered as an aggregated effect of these trading strategies.
Since there are numerous speculations happening in every minute, it
is natural to think that their aggregation is uncountable.
Later, we will show that an uncountable collection of null sets is
not necessarily incorporated in the $\mathbb{P}$-null set of $\mathcal{F}$.
This means that the aggregated effect, the bubble, may have a positive
probability to occur, namely to appear in $\mathcal{Y}$.

To formalize the previous argument, let $A_{1}$, $A_{2}$,$\dots$
$\in\mathcal{S}$ be a sequence of pairwise disjoint sets. In order
to ensure that $\pi_{t}$ is a regular conditional distribution, the
$\sigma$-additivity condition needs to be satisfied:
\[
\pi_{t}\left(\cup^{\infty}_{i}A_{i}\right)=\sum_{i=1}^\infty \pi_{t}(A_{i})
\]
 for every $\omega\in\Omega\;\backslash\;\mathcal{N}(A_{i},\: i\geq1)$, where
$\mathcal{N}(A_{i},\: i\geq1)$ is the $\mathbb{P}$-null set
for the disjoint set $A_{i}$ for any $i\geq1$. Let the collection
of these null sets be $\mathcal{N}_{0}$. Note that the power set
of all null sets is $2^{\mathbb{N}}$ which is uncountable.
This means that $\mathcal{N}_{0}$ is uncountable. We know that $\pi_{t}$
satisfies the $\sigma$-additivity condition only if $\omega\in\mathcal{N}(A_{i})$
for any $i\geq1$ but not $\omega\in\mathcal{N}_{0}$. Therefore,
some event in $\mathcal{N}_0\;\backslash\;\cup_{i=1}^\infty\mathcal{N}(A_{i})$
is not in the null sets for $\pi_{t}$ and has positive probability to occur:
\[
\mathcal{Y}\cap\{\mathcal{N}_{0}\;\backslash\;\cup_{i=1}^\infty\mathcal{N}(A_{i})\}\neq\varnothing.
\]
In fact, the set $\mathcal{N}_{0}$ need not even be measurable because
it is defined in terms of an uncountable union.\footnote{This statement follows from the axiom of choice, which allows for the construction of non-measurable sets, i.e., collections of events that do not have a measure in the ordinary sense, and whose construction requires an uncountable number of events.}
Then $\pi_{t}$ cannot
be a probability measure. The purpose of Theorem \ref{thm:ExistenceConditionalProb}
is to regularize this problem so that the projected $\pi_{t}$ is
on a countable subspace. This regularization implicitly forces $\pi_{t}$
to ignore those collections of countable $\mathbb{P}$-null sets on $\mathcal{F}$.
As a consequence, the abstract economic model might not be a ``proper'' model for all events, but one that approximates
a complex reality.

\section{A Feasible Econom(etr)ic Model}\label{sec:Model}

As shown in the previous section, the $\mathbb{P}$-null set on $\mathcal{F}$
may induce the arbitrariness of $\pi_{t}(\omega)$ on $\mathcal{Y}$.
For the $\mathbb{P}$-null set on $\mathcal{Y}$, individuals may have
arbitrary beliefs about the event sets, because they cannot figure
out any ``law'' on the set. The arbitrariness allows us to modify
$\mathcal{Y}_{t}$-adapted processes by changing the values of these
processes on the $\mathbb{P}$-null set, which corresponds to a change of measure.
Then the new process should still be $\mathcal{Y}_{t}$-adapted. It
accommodates the complexity of the real world but it induces a
class of arbitrary filters $\pi_{t}$. Due to the arbitrariness, the conditional
distribution process $\pi_{t}\mbox{(\ensuremath{\omega})}$ exists
even though some observable event sets in the economy are not explained
by the underlying model. If the model needs
a regular solution, it should be disencumbered of these irregularities.
In this section, we look for a feasible model that will regularize
the expected process $^{o}\varphi(X)$ and exploit a specific representation
of it.

With three additional claims, one can obtain an explicit solution
rather than an abstract process of $^{o}\varphi(X)$. These claims
are the following: First, the martingale fairness claim regularizes a class of probabilities
that are not uniquely defined on the $\mathbb{P}$-null set on $\mathcal{F}$.
Second, the invariance fairness claim induces a specification of $X$ that is
embedded in the general model $(\Omega,\mathcal{F},\mathbb{P})$.
Finally, the independent complement claim specifies the motions of the observable
process $Y$. The first and second claims basically consider the same
issue of finding a feasible sub-class models of the underlying economy
$(\Omega,\mathcal{F},\mathbb{P})$, but the development of the invariance
fairness claim depends on the martingale fairness claim. With the
specification of the law of $X_{t}$, the last claim induces a feasible
representation of $^{o}\varphi(X_{t})$ based on the observable process
$Y$.

\subsection{Fairness Existence}\label{sec:MF}

The following claim introduces a ``stochastic constant'' upon which
we can build our model:
\begin{claim*}
[Martingale Fairness, MF%
\footnote{The problem could be extended to a semi-martingale problem by using
a No Free Lunch claim (the Kreps-Yan Theorem). But then $X$ in general
cannot provide any explicit solution for the conditional probability
$\pi_t$. %
}] A probability measure $\mathbb{Q}$ on $(\Omega,\mathcal{F})$
is absolutely continuous with respect to $\mathbb{P}$, such that
$\mathbb{Q}\sim\mathbb{P}$. The information of state $X_{t}$ at
any time $t$ is ``fair'' for all agents under $\mathbb{Q}$ and
the information is memoryless, i.e. the process $X_{t}$ is Markovian%
\footnote{Formally, $\mathbb{E}[f(X_{T})|\mathcal{F}_{t}]=\mathbb{E}[f(X_{T})|X_{t}]$
for any $f(\cdot)\in B(\mathbb{S})$.%
}.
\end{claim*}
Fairness means the martingale property of $X$:
\[
\mathbb{E}_{\mathbb{Q}}[X_{t}|\mathcal{F}_{s},s\leq t]=X_{s}\quad\mbox{and}\quad\mathbb{E}_{\mathbb{Q}}[X_{t}-X_{s}|\mathcal{F}_{s},s\leq t]=0.
\]
The martingale model $(\Omega,\mathcal{F},\mathbb{Q})$ is treated
as a \emph{ghost model} since fairness may never happen in reality. However, if one accepts the existence of this martingale
model, it will guide us to a feasible base-line model and help us
to solve the original problem. If there is a $\mathbb{P}$-martingale process
$Z$ on $(\Omega,\mathcal{F})$, then any $\mathbb{Q}$-martingale
process $X$ implies a $\mathbb{P}$-martingale process $ZX$, due to the absolute continuity of $\mathbb{Q}$
and $\mathbb{P}$. It is obvious that if a process can be regularized on either measure,
then it can also be regularized on the other one.

The Markovian structure of $X$ means that the filtration $\mathcal{F}_{s}$
is independent of the $\mathcal{F}$-adapted $X_{u}$ if $s<t<u$.
For arbitrary time $t<u$, the Markovian structure implies a transition
kernel $\mathbb{Q}_{u-t}(X_{u}|X_{t})$. The Chapman-Kolmogorov equation
of the transition kernel is also available such that
\[
\mathbb{Q}_{u-s}(X_{u}|X_{s})=\int\mathbb{Q}_{u-t}(X_{u}|X_{t})\mathbb{Q}_{t-s}(dX_{t}|X_{s})
\]
which can be simply stated as $\mathbb{Q}_{\tau+\tau'}(\cdot|\cdot)=\mathbb{Q}_{\tau'}\mathbb{Q}_{\tau}$
for $\tau=t-s$, $\tau'=u-t$. The existence of the kernel $\mathbb{Q}_{\tau'}(\cdot|\cdot)$
is a direct result of the Kolmogorov existence theorem \citep[Theorem 7.4]{Kallenberg2002}.
It is obvious that the transition kernel $\mathbb{Q}_{\tau'}(\cdot|\cdot)$
is a regular conditional probability.

With the MF claim, in Section \ref{sec:Main-Results} Corollary \ref{cor:master-eq},
we give a gain-loss (master) equation to describe the dynamics of
$X_{t}$:
\[
\frac{\partial}{\partial\tau}\mathbb{Q}_{\tau}(X_{u}|X_{s})=\int\left\{ \mathcal{W}(X_{u}|X_{t})\mathbb{Q}_{\tau}(dX_{t}|X_{s})-\mathcal{W}(dX_{t}|X_{u})\mathbb{Q}_{\tau}(X_{u}|X_{s})\right\}
\]
where the function $\mathcal{W}(X_{u}|X_{t})$ is the time derivative
of the transition probability at $\tau=0$, called \emph{transition
probability per unit time}. This equation describes the complete transition
pattern of $X$ by showing the variation of the corresponding transition
kernel. If $\partial_{\tau}\mathbb{Q}_{\tau}(X_{u}|X_{s})$ is set
to zero, the evolution of $X$ attains a balance. The equation
merely states the fact that the sum of all transitions per unit time
into any state $X_{t}$ must be balanced by the sum of all transitions
from $X_{t}$ into other states. Gain balances loss, in other
words, we have a steady state.%
\footnote{When the process is assumed to be homogeneous in time, the family
of $\mathbb{Q}(\cdot|\cdot)$ is a semigroup of transition kernels
and has been extensively studied in recent works of operator methods,
see e.g. \citet{HansenSahaliaScheinkman2009}.%
}

\subsection{Invariance Behaviors}\label{sec:IF}

With the martingale fairness claim, we have seen that the Markovian
model gives us an equation to measure the variation of state transitions
of the underlying economy. The equation is valid at any time-point
and in any state, but the equation provides no clue about $\mathcal{W}(\cdot|\cdot)$,
the \emph{transition probability per unit of time}. Now an idea is to
extract some information about the statistics of $\mathcal{W}(\cdot|\cdot)$,
in particular first and second moments. This type of information should be able
to generate a class of sub-models that mimic the behavior of the original model of $X$. We need to find out under
which conditions
the sub-model is equivalent to the original one, in which case no
loss of information occurs when representing $(\Omega,\mathcal{F},\mathbb{P})$
by the ghost model $(\Omega,\mathcal{F},\mathbb{Q})$.

Let $f(\cdot,\cdot)$ be a function satisfying the \emph{maximum principle
up to second order}, which means that for a compact subset of states
$B\in\mathbb{S}$, at time $t$, the maximum of $f(t,x)$ in $x\in B$
is found on the boundary of $B$, $\partial B$. The simplest example
of $f$ is a function in the linear functional class such that for
fixed $t$, $x<x'\in B$ implies $f(t,x)<f(t,x')$ (or $>$), $\nabla_{x}f(t,x')\geq0$
(or $\leq$) and $\triangle_{x}f(t,x')=0$ on $B\subset\mathbb{R}$.
The extremum of $f(t,\cdot)$ always exists on the boundary of the
domain. Here $\triangle_{x}$ and $\nabla_{x}$ denote the Laplace
and gradient operators on $x$, respectively.

Think of $f(\cdot,\cdot)$ as a time-dependent utility or
value function. The requirement of $f(\cdot,\cdot)$ being maximal
up to second order means that $\triangle_{x}f(t,x)$ is proportional
to $\frac{\partial f}{\partial t}(t,x)$ so that one can set up their
relation by some equation, for example
\[
\frac{\partial f}{\partial t}(t,x)=-\frac{1}{2}\triangle_{x}f(t,x)
\]
which would imply that $X_{t}$ follows a Wiener process. Thus,
the maximum principle pins down a specific evolution class for $X_{t}$.
We have the following claim to incorporate this idea.
\begin{claim*}[Invariance Fairness, IF] If claim MF is true, then for any $f(t,x)$
satisfying the maximum principle up to second order, there exists a martingale measure such that $f(t,x)$
will preserve the fairness on this measure. The law of $X_{t}$
will also satisfy the maximum principle.
\end{claim*}
Theorem \ref{thm:MartinagleDiffusion} in Section 4 will show that the IF claim is another
way of specifying It\^o's diffusion problem%
\footnote{Mathematically, this claim intends to squeeze a stochastic problem
to a partial differential equation (PDE) problem so that it is possible
for economists to construct and solve a specific analytic problem. %
}. To the best of our knowledge, this is the first time that the problem is motivated
on the basis of the maximum principle.
Understanding the connection between this economic claim and econometric
models will help us to assess the potentials of modeling. That is, before doing
estimation, testing, or prediction, it is essential to realize
how far the model can reach in principle.

The diffusion structure induces a Wiener process specification for
$\mathcal{W}(\cdot|\cdot)$. The first and second moments of the
process are given by
\[
\int_{-\infty}^{\infty}x\mathcal{W}(X_{t}|dx),\quad\int_{-\infty}^{\infty}x^{2}\mathcal{W}(X_{t}|dx),
\]
 where $\mathcal{W}(\cdot)$ is the transition probability per unit
time under the Wiener law. This is a diffusion martingale type model.
Given the whole transition contents of $X$, our attention is only
restricted to those transitions that will maintain the maximum principle
up to second order. The reason is that only the transitions satisfying
\emph{invariance fairness} can be revealed and identified in standard
econom(etr)ic models. It does not mean that the unqualified transitions
do not exist. Conversely, many transitions in the system have high
order features such as complex trading strategies in pricing, multiple
correlated options, etc. What we can state however is that those transition
features are too complex to be embedded in a diffusion model%
\footnote{One can define a more complicated model to incorporate these effects,
but the cost is to use higher order stochastic calculus. In fact, later
we will see that the diffusion problem already induces an almost infeasible
representation for the conditional density. At this stage, the complexity
level of the problems that depart from the diffusion ones still
needs to be elaborated.%
}. Therefore, those higher order transition laws of $X$ will be
assigned to the $\mathbb{P}$-null set in $\mathcal{F}$.

\subsection{Indifferent Projection}\label{sec:IA}

The last component we have not yet exploited is the observable process $Y$.
In the economy, the process $Y$ reflects the law of $X$, so the
topological structure of $Y$ should contain as much information as
$X$. Since the IF claim is nothing but pinning the space of $X$ onto
the Wiener space $L^{2}(\mathcal{W})$, an $L^{2}$ space
with Wiener measure, it is natural to assume that $Y$ can be represented
in a similar space.

Given any map $h$ in $L^{2}$  and $Y=h(X_{t})$, if $Y$
can maintain all the information of the martingale diffusion process
of $X$, then we say that $Y$ shares an isometry property with $X$. Except
for the information maintained under the isometry property, note that
some information in $\mathcal{Y}$, such as the collection of $\mathbb{P}$-null sets in $\mathcal{F}$, is
not contained in $X$ but affects the outcome of $Y$. We use
measurement errors to represent this information. The following claim
is to specify the law of $Y$.

\begin{claim*}[Independent Complement, IC] Let $h(\cdot)$ be a map in $L^{2}$
which satisfies the maximum principle up to second order as in the IF claim. Suppose the observable process
$Y$ is contaminated by an additive generalized Wiener noise $W$, where
the noise process $W_{t}$ is generated by the information set $\mathcal{F}_{t}$
but is independent of $h(X_{t})$.
\end{claim*}

The information set of $W$ is generated by $\mathcal{F} \;\backslash\; \sigma(X)$ where
$\sigma(X)$ is the $\sigma$-algebra generated by those $X$ satisfying MF and IF claims.
In practice, the Wiener process is also modeled independently\footnote{The dependence between $W$ and $X$ is difficult to eliminate in
economics and may cause an endogeneity problem. But, technically
speaking, this issue often arises by using a too simple function
$h(\cdot)$. Since $h(\cdot)$ can be highly non-linear, i.e.
containing all endogenous effects, it is reasonable to ignore this
issue here. %
} of $X_{t}$. Thus $\mathcal{Y}_{t}$ is a larger filtration than
$\mathcal{F}_{t}$, i.e.
\[
\mathcal{Y}_{t}=\sigma(X_{s},W_{s},\: s\in[0,t])\vee\mathcal{N},
\]
 since it allows for the measurability of the noise process. The process
$Y$ satisfying the IC claim is given as follows:
\begin{equation}
Y_{t}=\int_{0}^{t}h(X_{s})ds+W_{t},\qquad t\geq0.\label{eq:ObservableProcess}
\end{equation}
 Note that this specification is to restrict the process $Y$ in $L^{2}(\mathcal{Y}_{t})$
because
\[
\mathbb{E}\left[\int h(X_{s})^{2}ds\right]<\infty,\quad\mbox{ and}\quad W_{t}\in L^{2}(\mathcal{Y}_{t}).
\]

Theorem \ref{thm:riskfree} in Section \ref{sec:Main-Results} implies
that for a class of these models, there will be a concrete way of
specifying the conditional expectation process of this class. This
theorem is an important step to derive a specific form for the filter.
The representation
of $X_{t}$, as a result of IF claim, induces a feasible conditional
expectation for $\varphi(X)$, while IC claim allows us to attain
the expectation of $\varphi(X)$ conditioning on the information generated
by $Y_{t}$.

IC claim has a similar role as MF claim. MF claim is to ensure
the existence of a martingale problem for the state process. IC claim
does the same but for the observable process. The aim is to make
the state process $X_{t}$, the diffusion generator in IF and the
observable process $Y_{t}$ comparable.

\subsection{An Explicit Representation}

The previous claims are to obtain a representation of the conditional
distribution $\pi_{t}$ for
our class of models. Once a representation of $\pi_{t}$ is available,
each model in this class will correspond to a specification of this
representation. The data contained in the information $\mathcal{Y}$
will be useful for estimating the parameters of this specification.
A specified representation plays a role as a predictor for the corresponding
model and observable information.

The filtering problem is, essentially, to determine the conditional
distribution $\pi_{t}$ of $X(\omega)$ at time $t$ given the
information accumulated from observing $Y$ in the interval $[0,t]$.
Given all three necessary claims, we show in Theorem \ref{thm:KSP}
that for any bounded continuous choice function $\varphi\in C_{b}(\mathbb{S})$,
we can compute the conditional expectation of $\varphi$
\begin{equation}
\pi_{t}(\varphi):=\mathbb{E}[\varphi(X_{t})|\mathcal{Y}_{t}],\label{eq:conditionalexp}
\end{equation}
via an equation called Kushner-Stratonovich-Pardoux equation.

 Many dynamical estimates consist of computing the conditional
distribution of a target process given a partially observed history.  As the explicit solution of
$\pi_{t}(\varphi)$ gives the end time marginals of the preceding conditional distributions defined
for any bounded $\varphi$, this explicit representation of $\pi_{t}(\varphi)$ provides a concrete basis of nonlinear estimation problems.
This point of view is also at heart of the Bayesian methodology, where
the conditional distribution is the posterior and the path distribution of the states is
the prior.

\section{Main Results\label{sec:Main-Results}}

The state $X_{t}$
is $\mathcal{F}_{t}$-adapted, while the constructed conditional expectation is evaluated by $\pi_{t}\left(X_{t}(\omega)\in A\right)=\mathbb{P}\left[X_{t}\in A|\mathcal{Y}_{t}\right](\omega)$, an $\mathcal{Y}_{t}$-adapted process. Thus $\mathbb{E}[\varphi(X_{t})|\mathcal{Y}_{t}]=\int\varphi(x)\pi_{t}(x)$
may not be well-defined. Let $^{o}\varphi(X_{t})$ be a counterpart of $\varphi(X_t)$
that is projected on the smallest $\sigma$-algebra on $([0,\infty)\times\Omega,\mathcal{B}([0,\infty))\otimes\mathcal{F})$
such that $^{o}\varphi(X_{t})$ is $\mathcal{Y}_{t}$-adapted and measurable. Our first result is to show that for any choice function $\varphi \in B(\mathbb{S})$, the conditional expectation of $\varphi(X_t)$ is equivalent to $^{o}\varphi(X_{t})$ in probability. This result implies the existence of filters in our abstract economy.

With an enlarged $\sigma$-algebra, a representative process will
be defined for $\varphi(X_{t})$ even if $\varphi(X_{t})$ is not $\mathcal{Y}_{t}$-adapted
\citep[Theorem 7.1]{RogersWilliams2000}. This theorem is called \emph{projection
theorem} and will be used in the proof of Theorem \ref{thm:ExistenceConditionalProb}.
The projection theorem says that if a process $X$ is measurable and
bounded, then for every stopping time $T$, there is a representation
$^{o}X$ (optional process) such that
\begin{equation}
^{o}X_{T}\mathbb{I}_{\{T<\infty\}}=\mathbb{E}[X_{T}\mathbb{I}_{\{T<\infty\}}|\mathcal{Y}_{T}],\label{eq:Optional}
\end{equation}
as a projection of $X$ onto $\mathcal{Y}$ where $\mathbb{I}_{\{A\}}$ is an indicator function for a set $A$.
Here no restriction is imposed on the stopping time $T$. The idea of projecting an $\mathcal{F}$-measurable element onto $\mathcal{Y}$ is similar to formulating a filter of $X$ given the observable information in $\mathcal{Y}$. We apply this result to show the existence of a filter in the abstract economy.

\begin{thm}
\label{thm:ExistenceConditionalProb}Let $\mathcal{P}(\mathbb{S})$ denote the space of all probability
measures on $\mathbb{S}$. For a compact set $\mathbb{S}$
and its Borel $\sigma$-algebra $\mathcal{\mathcal{S}}$, there is
a $\mathcal{P}(\mathbb{S})$-valued conditional distribution process
$\pi_{t}$ such that for any bounded $\mathcal{S}$-measurable function
$\varphi\in B(\mathbb{S})$,
\[
\mathbb{P}\left[\int_{\mathbb{S}}\varphi(x)\pi_{t}(dx)={}^{o}\varphi(X_{t})\quad\forall t\right]=1.
\]
This distribution process is an equilibrium process for the economy whose underlying states are in a
probability space $(\Omega,\mathcal{F},\mathbb{P})$ and whose observable information is contained in $\mathcal{Y}$.
\end{thm}
Theorem \ref{thm:ExistenceConditionalProb} implies the existence of $\pi_{t}$ for the abstract model given Section \ref{sec:The-Model}. The importance is that a perception of any bounded choice function always exists in this abstract economy although there could be multiple ways of forming the perception due to the incompatibility between underlying and observable layers of the economy.

Following the model in Section \ref{sec:The-Model}, we will specify a uniquely representable law of perception by using the claims in Section \ref{sec:Model}. First, the martingale fairness (MF) claim in Section \ref{sec:MF} regularizes a class of probabilities of these processes
that are not uniquely defined on the $\mathbb{P}$-null set on $\mathcal{F}$
so that the analysis of the model can rely on \emph{one} ghost model
$(\Omega,\mathcal{F},\mathbb{Q})$. The claim also discloses that
the evolution of the state $X$ is completely captured by the transition kernel $\mathbb{Q}_{\tau'}(\cdot|\cdot)$,
whose variation describes the variation of
the evolution pattern of $X$. Thus, the MF claim extracts important characteristics
of the underlying dynamics.
\begin{cor}
\label{cor:master-eq}The martingale model $(\Omega,\mathcal{F},\mathbb{Q})$
implies a gain-loss equation for the system such that:
\[
\frac{\partial}{\partial\tau}\mathbb{Q}_{\tau}(X_{u}|X_{s})=\int\left\{ \mathcal{W}(X_{u}|X_{t})\mathbb{Q}_{\tau}(dX_{t}|X_{s})-\mathcal{W}(dX_{t}|X_{u})\mathbb{Q}_{\tau}(X_{u}|X_{s})\right\} .
\]
The first term is the gain of state $X_{u}$ due to transitions from
other states $X_{t}$ and the second term is the loss due to transitions
from $X_{u}$ into other states.
\end{cor}
The equation in Corollary \ref{cor:master-eq} needs a further specification
because of the unspecified $\mathcal{W}(X_{u}|X_{t})$. With the invariance fairness (IF)
claim in Section \ref{sec:IF} we can obtain a specification within the It\^o diffusion problem.
The following theorem gives the equivalence.
\begin{thm}
\label{thm:MartinagleDiffusion}For $X_{t}\in(\mathbb{S},\mathcal{S})$
and $f(t,\cdot)\in C_{b}^{\infty}$, the following are equivalent:

(i) If claim IF is true, any $f(t,X_{t})$ in $C_{b}^{\infty}([0,\infty),\mathbb{S})$
has an approximating model that relies on the information contained
in the first two moments of the process $f(t,X_{t})$.

(ii) The function $f(t,X_{t})$ is an It\^o diffusion process with drift
and diffusion terms, $(a,b)=(a(X_{t}),b(X_{t}))$.
\end{thm}
For a diffusion type process, its first and second order moment
describe the full dynamics. Thus we can specify $\mathcal{W}$.
\begin{cor}
\label{cor:DiffusionCoeff}If $\mathbb{P}\in\mathcal{M}(\mathfrak{P}(\mathbb{S}))$,
then
\[
\left(f(X_{t})-f(X_{0})-\int_{0}^{t}(\mathbb{A}f)dt,\mathcal{PB}_{t},\mathbb{P}\right)
\]
 is a martingale, where $\mathbb{A}:=a(\cdot)\nabla_{x}+\frac{1}{2}b(\cdot)\triangle_{x}.$
In addition, if $a(\cdot)$ and $b(\cdot)$ are bounded and continuous,
the weak solution of the diffusion problem $(a,b)$ is unique. Then
\[
a(X_{t})=\int_{-\infty}^{\infty}x\mathcal{W}(X_{t}|dx),\quad b(X_{t})=\int_{-\infty}^{\infty}x^{2}\mathcal{W}(X_{t}|dx),
\]
 where the transition probability per unit time $\mathcal{W}(\cdot)$
has the Wiener law.
\end{cor}
By the property of characteristic functions, the martingale in Corollary
\ref{cor:DiffusionCoeff} together with the initial condition captures
all the information, the first and the second order moments, of $\mathcal{W}$.
This implies that $\mathbb{A}(\cdot)$ captures the first two moments
information of the process $f(X_{t})$ on the economic model $(\Omega,\mathcal{F},\mathbb{P})$.
Therefore the process $f(X_{t})$ is a diffusion type process on the
Wiener path.

In fact, the IF claim is nothing but pinning the problem onto the \emph{Wiener
space} $L^{2}(\mathcal{W})$, an $L^{2}$ space with Wiener measure.
The \emph{martingale representation theorem} says that any continuous
martingale, i.e.
\[
M_{f,t}:=\left(f(X_{t})-f(X_{0})-\int_{0}^{t}(\mathbb{A}f)dt,\mathcal{F}_{t},\mathbb{P}\right)
\]
generated by $\mathcal{W}$, can be written as
\[
M_{f}=\mathbb{E}[M_{f}]+\int_{0}^{T}h_{s}dW_{s}
\]
with a predictable process $h_{s}$
i.e. each $h_{s}$ is $\mathcal{Y}_{t}$-measurable for $t<s$. Without loss of generality, we
consider the case $\mathbb{E}[M_{f}]=0$. The functional space of
$h_{s}$ is
\[
L_{T}^{2}:=\left\{ h_{s}:h_{s}\mbox{ is }\mathcal{F}_{t}\mbox{-predictable and }\mathbb{E}\left[\int_{0}^{T}\|h_{s}\|^{2}ds\right]<\infty\right\} .
\]
 The stochastic integral of $h$ is a map $J:L_{T}^{2}\rightarrow L^{2}(\mathcal{F}_{T})$
such that
\[
J(h)=\int_{0}^{T}h_{s}dW_{s}.
\]
 This map is an isometry as a consequence of the It\^o isometry theorem.
The image of $J$ of the Hilbert space $L_{T}^{2}$ is complete. Therefore,
the martingale $M_{f}$ and the stochastic integral $J(h)\in L^{2}(\mathcal{F}_{t})$
are isometric.

What we emphasize here is that the IF claim carries us to an $L^{2}$
space where the classical projection techniques are available.

\begin{thm}
\label{thm:riskfree} If the MF and IF claims hold, then the IC claim in Section \ref{sec:IA}  implies
the representation (\ref{eq:ObservableProcess}) for the observable process $Y_{t}$. Suppose
$\mathbb{E}\left[\exp\left(\frac{1}{2}\int h(X_{s})^{2}ds\right)\right]<\infty,$
then the following statements are true:

(i) There exists a measure $\tilde{\mathbb{P}}$ such that
\[
\left.\frac{d\tilde{\mathbb{P}}}{d\mathbb{P}}\right|_{\mathcal{F}_{t}}=\exp\left(-\int_{0}^{t}h(X_{s})dW_{s}-\frac{1}{2}\int_{0}^{t}h(X_{s})^{2}ds\right),
\]
and, under measure $\tilde{\mathbb{P}}$, $Y$ is independent of $X$.
In addition, the motions of $X$ under $\tilde{\mathbb{P}}$ and
under $\mathbb{P}$ are the same.

(ii) For any $\mathcal{F}_{t}$-measurable random variable $\varphi(X)$,
\[
\tilde{\mathbb{E}}\left[\varphi(X)|\mathcal{Y}_{t}\right]=\tilde{\mathbb{E}}\left[\varphi(X)|\mathcal{Y}\right]
\]
 where $\mathcal{Y}=\vee_{t\in\mathbb{R}^{+}}\mathcal{Y}_{t}$ and
$\mathcal{Y}_{t}=\sigma(Y_{s},s\in[0,t])\vee\mathcal{N}$.
\end{thm}
The time-invariant algebra in Theorem \ref{thm:riskfree} enables us to
use techniques based on Kolmogorov's conditional expectation which
would not be applicable if the conditioning set was time dependent,
such as $\mathcal{Y}_{t}$.

With the existing results, we summarize the model specification in Section \ref{sec:Model} as
the following pair $(X,Y)$: $X$ is a solution of the martingale problem for $(\mathbb{A};\pi_{0})$;
in other words, assume that the distribution of $X_{0}$ is $\pi_{0}$
and that the process $M_{f}=\{M_{f,t},t\geq0\}$, where
\begin{equation}
M_{f,t}=f(X_{t})-f(X_{0})-\int_{0}^{t}\mathbb{A}f(X_{s})ds,\qquad t\geq0,\label{eq:martingaleRepresentation}
\end{equation}
 is an $\mathcal{F}_{t}$-adapted martingale for any $f\in C_{b}^{\infty}$
and $(\mathbb{A}f)(\cdot)$ corresponds to $(a(\cdot),b(\cdot))$
of a diffusion process. $Y$ satisfies the evolution equation
\[
Y_{t}=\int_{0}^{t}h(X_{s})ds+W_{t},\qquad t\geq0,
\]
with null initial condition.

Then our attempt is to connect the martingale problem in (\ref{eq:martingaleRepresentation})
with a diffusion type representation. Theorem \ref{thm:MartinagleDiffusion}
tells us that when the process is on the Wiener path, the solution
of a martingale problem associated with the second order differential
operator is the solution of the diffusion process. Theorem \ref{thm:riskfree}
tells us that $Y$ is on the Wiener path under $\tilde{\mathbb{P}}$.

Finally, we reach a specific representation of the perceptions
in the abstract economy.

\begin{thm}
\label{thm:KSP}(Kushner-Stratonovich-Pardoux, KSP) For any $\varphi\in C_{b}(\mathbb{S})$,
Proposition \ref{pro:Zakai} implies
\begin{equation}
\pi_{t}(\varphi)=\pi_{0}(\varphi)+\int_{0}^{t}\pi_{s}(\mathbb{A}\varphi)ds+\int_{0}^{t}\left(\pi_{s}(\varphi h)-[\pi_{s}(h)]^{2}\right)(dY_{s}-\pi_{s}(h)ds).\label{eq:KSP}
\end{equation}
where $\pi_{t}$ is the equilibrium density process for our general filtering setting. The conditional expectation
$\pi_{t}(\varphi) = \mathbb{E}[\varphi(X_{t}) |\mathcal{Y}_{t}]$ varies accordingly to \eqref{eq:KSP}
\end{thm}

Equation (\ref{eq:KSP}) is called KSP which has recently been applied to solve non-linear
filtering and smoothing problems in applied mathematics, see \citep{Bensoussan2004}.
One can think of the KSP representation as characterizing an \emph{equilibrium
conditional expectation} over any $\varphi\in C_{b}(\mathbb{S})$.
It is a stochastic PDE problem and has a unique solution.\footnote{See Chapter 4.8 of \citet{Bensoussan2004} for details
about the SPDE problem.}

Theorem \ref{thm:KSP} is a rather general result. Although solving
the KSP problem can be transferred to solving a parabolic PDE problem,
except for the case of a linear model and Gaussian disturbances and initial conditions,
finding a closed form expression for the distribution functions
of (\ref{eq:KSP}) can be very demanding.

\section{Remarks}

We give respective remarks regarding the previous claims, the modeling procedure within the general framework, and the relation between general filtering results and the linear ones.

\subsection{Claims in Asset Pricing Models}
The three claims of the previous section have their counterparts in asset pricing models. For illustration purposes, we only consider a simple situation where $X$ and $Y$ are measurable and observable.

Let $X$ be the price for some security contingent on an underlying asset $S$. Suppose that the
price at time $t$ is a random variable $X_t=\int_{0}^{t}H_{s}dS_{s}$,
where the integral is the It\^o integral and $H_t$ is predictable, i.e. each $H_t$ is ${\cal F}_s$-measurable for $s<t$.
The ``fairness'' in MF says that any $X$ constructed in this
way will have zero expected pay-off for some discounted price process under a probability measure $\mathbb{Q}$ such that $\mathbb{E}_{\mathbb{Q}}[\exp(-\int_t^T r_s ds)(X_T-X_t) | {\cal F}_t]=0$, where $r_t$ is the risk-free short rate process.
If this happens, by the fundamental theorem of asset pricing, the securities market admits no arbitrage.
$\mathbb{Q}$ is called the equivalent martingale measure.

Markov uncertainty is often assumed for diffusion processes. The claim IF pins down a specific transition of the security process as a diffusion process. This implies that the process $X$ is also a diffusion process. Suppose that $dX_{t}=\mu_{t}dt + \sigma_{t}dW_{t}$ where $(\mu_{t}, \sigma_{t})$ characterize the instantaneous drift and volatility, respectively, of this security. By Girsanov\rq{}s theorem, there is a risk-free measure $\tilde{\mathbb{P}}$ for the function $\varphi(X_{t})$ satisfying the maximal principle up to second order. In particular, if the market is complete, then for any diffusion process $Y$ one can obtain $Y$ by some self-financing strategy $h(\cdot)$ such that $Y_{t}=h(X_{t})$.

All consequences induced by these three claims, e.g. no-arbitrage,
diffusion path, and risk-free measures, are familiar to economists and should be acceptable for most dynamic models in
econom(etr)ics.

\subsection{Stochastic Volatility}
Consider an equity model with stochastic volatility as in \citet{ElliotSwishchuk2007}
\[dS_t = \mu S_tdt+z(X_t)S_tdW_t\]
where $S_t$ is the price of a stock, the function $z(\cdot)$ is known and $X_t$ is a hidden state Markov process. An example is the Heston model, where $z(x) = \sqrt{x}$ and
\[dX_t = \kappa(m-X_t)dt+\gamma\sqrt{X_t}dB_t.\] In general, if we observe a continuum of prices, then $z(X_t)$ is measurable with respect to the filtration generated by $\{S_\tau:\tau\leq t\}$. Let the observable process $Y_t=\log S_t$, and notice that\footnote{By It\^o calculus, $dY_{t}=d\log S_{t} = (1/S_t)dS_t - (1/2S_t^2) dt$. As $dS_t /S_t =\mu dt + z(X_t)dW_t$ and $dt/S_t^2 = z^2(X_t)$, we have the expression.}
\[dY_t = \left(\mu-\frac{1}{2}z^2(X_t)\right)dt+z(X_t)dW_t.\]
The noise $W_t$ of $Y_t$ contains a diffusion coefficient function $z(X_t)$. The suitable corresponding $\tilde{\mathbb{P}}$ for $Y$ is \[
\left.\frac{d\tilde{\mathbb{P}}}{d\mathbb{P}}\right|_{\mathcal{F}_{t}}=\exp\left(-\int_{0}^{t} \frac{\mu-z^2(X_t)/2}{z(X_t)} dW_{s}-\frac{1}{2}\int_{0}^{t}\left( \frac{\mu-z^2(X_t)/2}{z(X_t)} \right) ^{2}ds\right),\]
so that under $\tilde{\mathbb{P}}$ the process is independent of $X_t$.

We can discretize the model. For a fixed $t>0$, let $(t_k)_k$ be a partition of $[0,t]$, then the quadratic variation of $Y$ is the cumulative variance
\[[Y]_t = \lim_{\sup_k(t_{k+1}-t_k) \rightarrow 0}\sum_k(\Delta Y_{t_k})^2=\int_0^t z^2(X_\tau)d\tau\]
where $\int_0^t z^2(X_\tau)d\tau$ is $\mathcal Y_t$-measurable. If $z(X_t)$ is a continuous process, we have
\[\frac{d}{dt}[Y]_t = z^2(X_t)\]
that is also $\mathcal Y_t$-measurable. Thus if $z(X_t)$ is a continuous process, the volatility is observable for almost every $t$, and $X_t$ is observable if $z^{-1}$ exists.

With the Markov structure for $X_t$, one can price derivatives on $S_t$. The Black-Scholes price of a European call option $C_{BS}(t, S_t;T,K,Z_{[t,T]})$ in the presence of Markovian volatility is
\[
\mathbb{E}[C_{BS}(t, S_t;T,K,Z_{[t,T]})|\mathcal Y_t]=\int C_{BS}(t,S_t;T,K,x)\pi_t(dx),
\]
where $Z_{[t,T]} = \frac{1}{T-t}\int_t^T z^2(X_s)ds$, and $\mathbb E [~\cdot~]$ is w.r.t. the market's pricing measure.
Given $X_t$ and the parameters of its dynamics under the market measure, we can compute the expected return of the call option by taking a filtering expectation.

\subsection{Kalman Filter}

If $h(X_{t})$ and $f(X_{t})$ at every time $t$ can be linearized
as matrices (vectors) $\mathbf{H}_{t}X_{t}+\mathbf{h}_{t}$ and $\mathbf{F}_{t}X_{t}+\mathbf{f}_{t}$
such that
\begin{align}
X_{t}= & X_{0}+\int_{0}^{t}(\mathbf{F}_{s}X_{s}+\mathbf{f}_{s})ds+\int_{0}^{t}\sigma_{s}dV_{s},\label{eq:Xlinearization}\\
Y_{t}= & \int_{0}^{t}(\mathbf{H}_{s}X_{s}+\mathbf{h}_{s})ds+W_{t},\label{eq:Ylinerization}
\end{align}
then KSP with test functions $\varphi=x_{i}$ and $\varphi=x_{i}x_{j}$
will give us the standard Kalman filter, also called Kalman-Bucy filter, as follows.
Let $\hat{x}$ be the conditional mean of $X$ such that
\[
\hat{x}_{i,t}=\mathbb{E}[X_{i,t}|\mathcal{Y}_{t}]
\]
 and $R$ be the conditional covariance such that
\[
R_{t}^{ij}=\mathbb{E}[X_{i,t}X_{j,t}|\mathcal{Y}_{t}]-\mathbb{E}[X_{i,t}|\mathcal{Y}_{t}]\mathbb{E}[X_{j,t}|\mathcal{Y}_{t}].
\]
If (\ref{eq:Xlinearization}) and (\ref{eq:Ylinerization}) are acceptable
localizations for (\ref{eq:martingaleRepresentation}), then the solution
of $\hat{x}_{t}$ satisfies the following SDE
\begin{equation}
d\hat{x}_{t}=(\mathbf{F}_{t}\hat{x}_{t}+\mathbf{f}_{t})dt+R_{t}\mathbf{H}_{t}^{T}(dY_{t}-(\mathbf{H}_{t}\hat{x}_{t}+\mathbf{h}_{t})dt),\label{eq:KalmanMean}
\end{equation}
where we substitute $f(\hat{x}_t)=\mathbf{F}_t \hat{x}_t + \mathbf{f}_t$ and $h(\hat{x}_t) = \mathbf{H}_t \hat{x}_t +\mathbf{h}_t$ into KSP equation \eqref{eq:KSP} and where the covariance term $R_{t}$ satisfies the deterministic Riccati equation%
\footnote{A detailed proof is given in Theorem 4.4.1 of \citet{Bensoussan2004}.%
}
\begin{equation}
\frac{dR_{t}}{dt}=\sigma_{t}\sigma_{t}^{T}+\mathbf{F}_{t}R_{t}+R_{t}\mathbf{F}_{t}^{T}-R_{t}\mathbf{H}_{t}^{T}\mathbf{H}_{t}R_{t}.\label{eq:KalmanRic}
\end{equation}
Equation \ref{eq:KalmanMean} and \ref{eq:KalmanRic} together give the Kalman-Bucy filter scheme. One can see that this scheme is a special
case in the content of KSP equation.
While we feature the Kalman filter in this paper, there are other well known filtering methods including particle filters and the Zakai equation that are relating to KSP problem.

\section{Conclusion}

We have started with the fact that filtering is an intrinsic element of economic phenomena. For a general abstract economy, we provide
a result on the existence of filtering mechanisms. We emphasize a subtlety due to null sets that may lead to peculiar events with positive
probability after aggregation even though on an individual level such events have zero probability. This feature turns out to be crucial for the
understanding and interpretation of the economic model. It also has to be regularized in the derivation of the existence result.

By introducing three natural claims, we established a representation of the conditional distribution process and, hence, of the filtering device.
The general representation is nonlinear and subject to estimation using statistical methods. We have outlined the realm of economic models for which
this representation is applicable.  The implication of our findings for the way economic theory and econometrics interact in general has
yet to be discovered.

\appendix

\section{Appendix}

{\footnotesize{}\baselineskip19.15pt}{\footnotesize \par}

\subsection{Proof of Main Theorems}

\subsubsection{Proof of Theorem \ref{thm:ExistenceConditionalProb}}
\begin{proof}
The proof includes four steps: 1. construct a countable vector space
$\mathcal{U}$ on $\mathbb{S}$, 2. define a non-negative process
that corresponds to the elements in $\mathcal{U}$, 3. extend $\mathcal{U}$
to the space of continuous bounded functions, $C_{b}(\mathbb{S})$,
check that the definition of the process is still valid, and find a representation
of $\pi_{t}$, 4. extend $C_{b}(\mathbb{S})$ to $B(\mathbb{S})$
and check that all properties are still valid.

Step 1. For $C_{b}(\mathbb{S})$, compact $\mathbb{S}$ induces that
$C_{b}(\mathbb{S})$ is dense and that a linear span exists. Let $\{\varphi_{i}\}_{i=1}^{\infty}$
be the set of basis functions in the linear span and thus any $\varphi_{i}$
is bounded continuous. Let $\mathcal{U}$ be a countable vector space
generated by finite linear combinations of $\{\varphi_{1},\dots,\varphi_{n}\}$
with rational coefficients such that
\[
\mathcal{U}:=\left\{ \varphi=\sum_{i=1}^{n}\alpha_{i}\varphi_{i},\; a_{i}\mbox{ is rational for all }i\right\} .
\]
 These $\varphi_{i}$s are still linearly independent for any $i\in\mathbb{Z}^{+}$.
Set $\varphi_{1}=1$.

Step 2. For any $t$, $\varphi_{n}(X_{t})$ is another $\mathcal{F}_{t}$-adapted
process%
\footnote{This statement skips one intermediate step which requires $X_{t}$
to be progressively measurable (see \citet{Stroock2000} Remark
7.1.1 Lemma 7.1.2).%
}. Equation (\ref{eq:Optional}) implies that a $\mathcal{Y}_{t}$-adapted
optional process $g_{n}^{t}$ exists for $\varphi_{n}(X_{t})$. Thus
a sequence $\{g_{i}^{t}\}_{i=1}^{n}$ is corresponding to $\{\varphi_{i}(X_{t})\}_{i=1}^{n}$.
For some $N\in\mathbb{Z}$, linear independence induces that a function
$\varphi\in\mathcal{U}$ is uniquely represented by $\sum_{i=1}^{N}\alpha_{i}\varphi_{i}$
and furthermore it implies that a $\mathcal{Y}_{t}$-adapted process
$g^{t}$, corresponding to $\varphi$, is linearly and uniquely represented
by $\sum_{i=1}^{N}\alpha_{i}g_{i}^{t}$. We can define the linear
functional
\[
\Lambda_{\omega}^{t}(\varphi)=g^{t}(\omega),\quad\forall t.
\]

Because the conditional distribution is a non-negative process, we
need to construct a non-negative analog of $\Lambda_{\omega}^{t}$.
Define a subspace
\[
\mathcal{U}^{+}:=\{u\in\mathcal{U},\ u\geq0\}
\]
that is countable. For $u\in\mathcal{U}^{+}$ and fixed $t$, we define
the null set for $u$ such that
\[
\mathcal{N}(u):=\left\{ \omega\in\Omega:\:\Lambda_{\omega}^{t}(u)<0\right\} .
\]
Since $u\geq0$, in order to show that $\mathcal{N}(u)$ is a $\mathbb{P}$-null
set, we need to show $\Lambda_{\omega}^{t}(u)\geq0$ almost surely.
If $u(\omega)\geq0$ almost surely, then by equation (\ref{eq:Optional})
the optional process would be non-negative on $\mathcal{Y}_{T}$ and
hence $\mathcal{N}(u)$ is a $\mathbb{P}$-null set. The union of
$\mathcal{N}(u)$ over $\mathcal{U}^{+}$,
\[
\mathcal{N}:=\cup_{u\in\mathcal{U}^{+}}\mathcal{N}(u)
\]
is a countable union. A new process $\bar{\Lambda}_{t}^{\omega}$ is defined as
\[
\bar{\Lambda}_{t}^{\omega}(\varphi):=\begin{cases}
\Lambda_{t}^{\omega}(\varphi) & \quad\omega\notin\mathcal{N},\\
0 & \quad\omega\in\mathcal{N}.
\end{cases}
\]

Step 3. In order to extend the definition of $\bar{\Lambda}_{t}^{\omega}(\varphi)$
to $\varphi$ outside $\mathcal{U}$, we first need to check that $\bar{\Lambda}_{t}^{\omega}$
is bounded. It is obvious that $\bar{\Lambda}_{t}^{\omega}(1)=1$.
Since $\varphi\in\mathcal{U}$, the uniform norm has the property
that $|\varphi|\leq\|\varphi\|_{\infty}1$. Then $\|\varphi\|_{\infty}1\pm\varphi\geq0$,
from step 2, we know
\begin{align*}
\bar{\Lambda}_{t}^{\omega}(\|\varphi\|_{\infty}1\pm\varphi & )\geq0\\
\|\varphi\|_{\infty}\pm\bar{\Lambda}_{t}^{\omega}(\varphi) & \geq0
\end{align*}
where the second inequality comes from the linearity of $\bar{\Lambda}_{t}^{\omega}$
and $\bar{\Lambda}_{t}^{\omega}(1)=1$. It implies
\[
\sup_{t}\|\bar{\Lambda}_{t}^{\omega}(\varphi)\|_{\infty}<\|\varphi\|_{\infty},
\]
so that $\bar{\Lambda}_{t}^{\omega}$ is bounded.

Let any $\varphi\in C_{b}(\mathbb{S})$. Since $\mathcal{U}$ is dense
in $C_{b}(\mathbb{S})$, there exists a sequence $\varphi_{k}\in\mathcal{U}$
such that $\varphi_{k}\rightarrow\varphi$. We can define
\[
\tilde{\Lambda}_{t}^{\omega}(\varphi):=\begin{cases}
\bar{\Lambda}_{t}^{\omega}(\varphi) & \quad\varphi\in\mathcal{U},\\
\lim_{k}\Lambda_{t}^{\omega}(\varphi_{k}) & \quad\varphi\in C_{b}(\mathbb{S})\setminus\mathcal{U}
\end{cases}
\]
over $C_{b}(\mathbb{S})$. For boundedness, we only need to
check the case $\varphi\in C_{b}(\mathbb{S})\setminus\mathcal{U}$.
Note that for any two sequences $\varphi_{k}$ and $\varphi_{j}$,
if $\varphi_{k}\rightarrow\varphi$ and $\varphi_{j}\rightarrow\varphi'$,
we will have
\[
\sup\|\tilde{\Lambda}_{t}^{\omega}(\varphi_{k})-\tilde{\Lambda}_{t}^{\omega}(\varphi_{j})\|_{\infty}\leq\|\varphi_{k}-\varphi\|_{\infty}+\|\varphi-\varphi'\|_{\infty}+\|\varphi'-\varphi_{j}\|_{\infty}
\]
by the boundedness result in $\mathcal{U}$ and the triangle inequality.
Thus, $\tilde{\Lambda}_{t}^{\omega}(\varphi)$ is bounded.

We also need to ensure that the optional process of $\tilde{\Lambda}_{t}^{\omega}(\varphi)$
is well-defined on $C_{b}(\mathbb{S})$. For $\varphi_{k}$ in $\mathcal{U}$,
we have a $\mathcal{Y}_{t}$-adapted process $\tilde{\Lambda}_{t}^{\omega}(\varphi_{k})$
for $\varphi_{k}(X_{t})$, and
\begin{align*}
\mathbb{E}\left[\tilde{\Lambda}_{T}^{\omega}(\varphi)\mathbb{I}_{T<\infty}\right]=& \; \lim_{k\rightarrow\infty}\mathbb{E}\left[\tilde{\Lambda}_{T}^{\omega}(\varphi_{k})\mathbb{I}_{T<\infty}\right],\\
=& \;\lim_{k\rightarrow\infty}\mathbb{E}\left[\varphi_{k}(X_{T})\mathbb{I}_{T<\infty}\right],\\
=& \; \mathbb{E}\left[\varphi(X_{T})\mathbb{I}_{T<\infty}\right].
\end{align*}
The last equation is implied by the dominated convergence theorem for
bounded sequences.

Since $\mathbb{S}$ is compact, the Riesz representation theorem implies
the existence of $\pi_{t}^{\omega}$,
\[
\tilde{\Lambda}_{T}^{\omega}(\varphi)=\int_{\mathbb{S}}\varphi(x)\pi_{t}^{\omega}(dx)=\left\langle \pi_{t}^{\omega},\varphi\right\rangle =\pi_{t}^{\omega}\varphi,\qquad\mbox{for }\forall t
\]
for any bounded and well-defined inner product.

Step 4. The last step is to extend the definition of $\pi_{t}^{\omega}\varphi$
to incorporate $\varphi\in B(\mathbb{S})$. Let $\bar{B}(\mathbb{S})$
be a subset of $B(\mathbb{S})$ such that $\pi_{t}^{\omega}\varphi$
is a $\mathcal{Y}_{t}$-adapted optional process of $\varphi(X_{t})$
on $\bar{B}(\mathbb{S})$. It is obvious that $C_{b}(\mathbb{S})\subset\bar{B}(\mathbb{S})$.
Note that the Borel $\sigma$-algebra generated by $B(\mathbb{S})$
is $\mathcal{B}(\mathbb{S})$. By the completeness of $C_{b}(\mathbb{S})$,
we can construct a sequence of subsets $\{\bar{B}_{i}(\mathbb{S})\}_{i}^{\infty}$
such that
\[
\bar{B}_{1}(\mathbb{S})\subset\bar{B}_{2}(\mathbb{S})\subset\cdots.
\]
Compactness of $\mathbb{S}$ implies that $\mathcal{B}(\mathbb{S})$
is closed under finite intersections. From the construction in step
1, we know that the constant function is included in every $\bar{B}_{i}(\mathbb{S})$.
The monotone class theorem implies $\cup_{i}\bar{B}_{i}(\mathbb{S})\supseteq\mathcal{B}(\mathbb{S})$,
since any monotone non-negative increasing sequence $\{\bar{B}_{i}(\mathbb{S})\}_{i}^{\infty}$,
with indicator function of every set in $\mathbb{S}$, contains the
$\sigma$-algebra $\mathcal{B}(\mathbb{S})$ which is closed under
finite intersections. Thus $\bar{B}(\mathbb{S})$ contains every bounded
$\mathcal{S}$-measurable function of $\mathbb{S}$. As $\bar{B}(\mathbb{S})$
is a subset of $B(\mathbb{S})$, we conclude $\bar{B}(\mathbb{S})=B(\mathbb{S})$.
\end{proof}

\subsubsection{Proof of Theorem \ref{thm:MartinagleDiffusion}}
\begin{proof}
From (ii) to (i), the proof is trivially applying It\^o's calculus.

From (i) to (ii), the proof consists of the following four steps:
1. show that the maximum principle on smooth functions is equivalent
to the law of Wiener processes, 2. show that the invariance of the law
is preserved on the Wiener path, 3. set up the approximation
on the Wiener path by showing that the martingale fairness is preserved,
and 4. extend the result to the model $(\Omega,\mathcal{F},\mathbb{P})$.

Step 1. The definition of the maximum principle is simply the first and second derivative conditions in
calculus. If a function $f:\mathbb{S}\rightarrow\mathbb{R}$
attains its maximum at point $x\in\mathbb{S}$, then
$\nabla_{x}f(x)=0$ and $\triangle_{x}f(x)\leq0$.
Furthermore, if $f$ is a time-dependent function such that $f:[0,T)\times\mathbb{S}\rightarrow\mathbb{R}$
at a certain time interval $[0,t]$, and $f$ attains its maximum at $x$
when time is $t$, then $\partial f(t,x)/\partial t \geq 0$ with $\nabla_{x}f(t,x)=0$ and $\triangle_{x}f(t,x)\leq0$.
The inequality
$\partial f(t,x)/\partial t \geq 0$ expresses the uncertainty of the future such that
$\partial f(\cdot,x)/\partial t $ could either strictly increase along $t$
or attain its optimum at $t$. Since the maximum principle is preserved up
to the second order, we have the heat equation
\begin{equation}
a(x) \frac{\partial f}{\partial t}(t,x) + \frac{b(x)}{2} \triangle_{x}f(t,x) = 0,\label{eq:heatequation}
\end{equation}
Without loss of generality, in steps 1 to 3, we only consider
the standard case with the diffusion factor $a(x)=b(x)=1$, but (\ref{eq:heatequation})
holds for any real vectors $a(x)$ and $b(x)$. The solution of (\ref{eq:heatequation})
is the well-known Wiener process.

Step 2: In order to formalize the concept of the Wiener path, we need
to introduce the path space. Suppose that a series of realizations
$\{x_{t_{i}}\}_{t_{i}\leq t_{N}}$ corresponds to $t$ via $x_{i}=\psi(t_{i})$
for $t_{i}\leq t_{N}$. Then $\psi:[0,\infty)\rightarrow\mathbb{S}$
is a continuous path with the image on the complete separable space
$\mathbb{S}$. A path space $\mathfrak{P}(\mathbb{S})=C([0,\infty),\mathbb{S})$
is a continuous function space of paths $\psi$. The $\sigma$-algebra
$\mathcal{PB}$ is
\[
\mathcal{PB}_{s}:=\sigma\left(\psi(t):t\in[0,s]\right),\quad s\in[0,\infty)
\]
generated by $\psi\in\mathfrak{P}(\mathbb{S})\mapsto\psi(t)\in\mathbb{S}$.
The measure $\mathcal{W}$ for $\mathfrak{P}(\mathbb{S})$ is called
the Wiener measure%
such that for a sequence $\{\psi(t_{i})\}_{t_{i}\leq t_{N}}=\{x_{t_{i}}\}_{t_{i}\leq t_{N}}$:
\begin{align*}
\mathcal{W}\left(\psi:x_{1}\in A_{t_{1}},\dots x_{t}\in A_{t_{N}}\right) & =\\
\int_{A_{t_{1}}}\cdots\int_{A_{t_{N}}}\frac{1}{\sqrt{2\pi(t_{1}-t_{0})}}e^{-\frac{(x_{1}-x_{0})^{2}}{2(t_{1}-t_{0})}}\cdots & \frac{1}{\sqrt{2\pi(t_{N}-t_{N-1})}}e^{-\frac{((x_{N}-x_{N-1})^{2}}{2(t_{N}-t_{N-1})}}dx_{t_{1}}\cdots dx_{t_{N}}.
\end{align*}
The measure is tight in the sense that, if $t-s<\epsilon$,
\[
\lim_{\epsilon\rightarrow0}\sup_{\psi\in\mathfrak{P}(\mathbb{S})}\sup_{0\leq s\leq t\leq T}\rho(\psi(t),\psi(s))=0
\]
for any metric $\rho(\cdot,\cdot)$. This is the Ascoli-Arzela criterion
for compact subsets.

We need to show that the invariance property of $\mathcal{W}$ is
a restatement of the independent identical increment property.

\emph{Identical}: Note that a function $f$ over $\psi$ will not
change the expression except that $\psi(t)$ is replaced by $f(\psi(t))$.
By Lemma 3.4.3 and Theorem 3.4.16 (Kolmogorov's Criterion) of \citet{Stroock2000},
we have that for a subset $\mu$ of all tight measures $\mathcal{M}(\mathfrak{P}(\mathbb{S}))$
and $\psi\in\mathfrak{P}(\mathbb{R})$:
\[
\sup_{\mu\in\mathcal{M}(\mathfrak{P}(\mathbb{S}))}\mathbb{E}_{\mu}\left[|\psi(t)-\psi(s)|^{r}\right]\leq C_{T}|t-s|^{1+\alpha},
\]
where $C_{T}<\infty$ is a constant, $\alpha>0$ and $r\geq1$. Then
we have
\[
\lim_{t\rightarrow s}\sup_{\psi\in\mathfrak{P}(\mathbb{S})}\frac{(\psi(t)-\psi(s))^{2}}{(t-s)}=\lim_{t\rightarrow s}\sup_{\psi\in\mathfrak{P}(\mathbb{S})}\left(\frac{\psi(t)-\psi(s)}{t-s}\right)^{2}(t-s)\rightarrow0.
\]
This means that the increments are controlled by the length of the time interval.
When the interval is extremely small, all increments are essentially treated
the same. So the smooth function $f$ does not matter for the law of $\mathcal{W}$.

\emph{Independent}: For $\psi$, $\varpi\in\mathfrak{P}(\mathbb{R})$,
let $\varpi(t)=\psi(t+s)-\psi(s)$. By the definition of the Wiener measure,
both $\psi(s)$ and $\varpi(t)$ associate with $\mathcal{W}$ on
the time path $[0,s]$ and $[0,t]$ respectively. Clearly, they are
independent.

Step 3. The reason why we are looking for a martingale representation
is in fact to look for a ``stochastic constant''. In the deterministic
case, suppose we define an integral curve of $\psi(\cdot)$ on a smooth
vector field $a$ on $\mathbb{R}$, starting at $x\in\mathbb{R}$.
Then the path $\psi$ with $\psi(0)=x$ has the property that
\[
f(\psi(t))-\int_{0}^{t}\langle a,\nabla_{x}f\rangle(\psi(\tau))d\tau
\]
is a \emph{constant}%
\footnote{The constant is the initial value $\psi(0)=x$ from the following
ODE problem:
\[
\frac{\partial f(\psi(t))}{\partial t}=\langle a,\nabla f\rangle(\psi(\tau)).
\]
}\emph{ }for any $f\in C^{\infty}$. If there is a stochastic analog,
then we can use this \emph{stochastic constant} to establish our approximating
model. The aim is to maintain a stable ``error''.

Recall the path space $\mathfrak{P}(\mathbb{S})$ and its $\sigma$-algebra
$\mathcal{PB}$. For an incremental element $\psi(t)-\psi(s)$ on $\mathfrak{P}(\mathbb{R})$,
the Fourier transform is:
\[
\mathbb{E}_{\mathcal{W}}\left[e^{i\xi(\psi(t)-\psi(s))}|\mathcal{PB}_{s}\right]=\int e^{i\xi x}\frac{1}{\sqrt{2\pi(t-s)}}e^{-x^{2}/2(t-s)}dx=e^{-\frac{|\xi|^{2}}{2}(t-s)}
\]
where $x=\varpi(t-s)=\psi(t)-\psi(s)$. What we want to obtain is a martingale
and a ``constant'' under $\mathcal{W}$. From the above equation,
it easy to see that we can obtain both of them simultaneously if we
shift the element $\exp i\xi\psi(t)$ by a Gaussian factor $\exp|\xi|^{2}t/2$:
\begin{align*}
\mathbb{E}_{\mathcal{W}}\left[e^{i\xi\psi(t)}e^{\frac{1}{2}|\xi|^{2}t}|\mathcal{PB}_{s}\right]= & e^{\frac{1}{2}|\xi|^{2}t}\mathbb{E}_{\mathcal{W}}\left[e^{i\xi\psi(t)-\psi(s)+\psi(s)}|\mathcal{PB}_{s}\right]\\
= & e^{\frac{1}{2}|\xi|^{2}t}e^{-\frac{|\xi|^{2}}{2}(t-s)}\mathbb{E}_{\mathcal{W}}\left[e^{i\xi\psi(s)}|\mathcal{PB}_{s}\right]\\
= & \mathbb{E}_{\mathcal{W}}\left[e^{i\xi\psi(s)}e^{\frac{1}{2}|\xi|^{2}s}|\mathcal{PB}_{s}\right]=1.
\end{align*}
Let a triplet denote this martingale on the Wiener path $\mathcal{W}$:
\begin{equation}
\left(\exp\left[i\xi\psi(t)+\frac{1}{2}|\xi|^{2}t\right],\mathcal{PB}_{t},\mathcal{W}\right).\label{eq:baseMartingale}
\end{equation}

We define the Fourier transform of $f$ by $\mathbb{F}f(\xi)=\int_{-\infty}^{\infty}f(x)e^{i\xi x}dx$,
and the inverse Fourier transform is $\mathbb{F}^{-1}f(\xi)=\int_{-\infty}^{\infty}f(x)e^{-i\xi x}dx$.

As in the deterministic case, the ideal representation of $f(t,\psi(t))$
on $\mathcal{W}$ is the path integral:
\[
\int_{0}^{t}\left[\nabla_{x}f+\frac{1}{2}\triangle_{x}f\right](\tau,\psi(\tau))d\tau.
\]
We need to check whether the approximation error is a ``constant''
in the stochastic sense. Note that
\[
f(t,x)=(2\pi)^{-1}\int e^{i(\xi t+\xi x)}(\mathbb{F}^{-1}f)d\xi d\eta.
\]
By the property $\mathbb{F}^{-1}(\frac{\partial}{\partial x})(\cdot)=i\xi\mathbb{F}^{-1}(\cdot)$,
we have
\[
\mathbb{F}^{-1}\left(\nabla_{x}f+\frac{1}{2}\triangle_{x}f\right)=\left(i\xi-\frac{1}{2}|\xi|^{2}\right)(\mathbb{F}^{-1}f).
\]
The approximating error is
\begin{align*}
f(t,\psi(t))- & \int_{0}^{t}\left[\nabla_{x}f+\frac{1}{2}\triangle f\right](\tau,\psi(\tau))d\tau\\
=(2\pi)^{-1}\int\int & \underset{M_{\xi}(t)}{\underbrace{\left[e^{i(\xi t+\xi\psi(t))}-\int_{0}^{t}e^{i(\xi\tau+\xi\psi(\tau))}(i\xi-\frac{1}{2}|\xi|^{2})d\tau\right]}}(\mathbb{F}^{-1}f)d\xi d\eta
\end{align*}
The Fourier term $\mathbb{F}^{-1}f$ is bounded and irrelevant for
$\mathcal{W}$. If $M_{\xi}(t)$ is a martingale in $\mathcal{W}$,
then the error will be a stochastic constant. Rewrite $M_{\xi}(t)$
as:
\[
M_{\xi}(t)=e^{i\xi t}e^{i\xi x}-\int_{0}^{t}e^{i\xi\psi(\tau)}e^{i\xi\tau}d(i\xi-\frac{1}{2}|\xi|^{2})\tau.
\]
The second term can be written as
\[
\int_{0}^{t}e^{i\xi\psi(\tau)+\frac{1}{2}|\xi|^{2}\tau}d(e^{i\xi\tau}\cdot e^{-\frac{1}{2}|\xi|^{2}\tau})
\]
and the first term can be written as $e^{i\xi t-\frac{1}{2}|\xi|^{2}t}e^{i\xi\psi(t)+\frac{1}{2}|\xi|^{2}t}$.
Fubini's Lemma together with (\ref{eq:baseMartingale}) implies that
\[
\mathbb{E}_{\mathcal{W}}\left[M_{\xi}(t)|\mathcal{PB}_{s}\right]=1\cdot\mathbb{E}_{\mathcal{W}}\left[e^{i\xi t-\frac{1}{2}|\xi|^{2}t}-\int_{0}^{t}d(e^{i\xi\tau-\frac{1}{2}|\xi|^{2}\tau})d\tau|\mathcal{PB}_{s}\right]=1.
\]
Thus $\left(f(t,\psi)-\int_{0}^{t}\left[\nabla_{x}f+\frac{1}{2}\triangle_{x}f\right](\tau,\psi)d\tau,\mathcal{PB}_{t},\mathcal{W}\right)$
is a martingale.

Now we consider the general case in (\ref{eq:heatequation}). If the state moves with velocity $a(X_{t})$, the path derivative
becomes $a(\cdot)\nabla f$. Moreover, the Laplace operator $\triangle$
in the heat equation (\ref{eq:heatequation}) may be associated with a volatility coefficient $b(\cdot)$.
Then the approximating model is given by
\[
\int_{0}^{t}\left[a(X_{s})\nabla_{x}f+\frac{1}{2}b(X_{s})\triangle_{x}f\right]ds,
\]
which is the integral of the Feller generator $\mathbb{A}$ on $f$:
\[
\mathbb{A}:=a(\cdot)\nabla_{x}+\frac{1}{2}b(\cdot)\triangle_{x}.
\]
The generator is a dual representation of a diffusion process $(a,b)$
such that
\[
dX_{t}=a(X_{t})dt+\sigma(X_{t})dV_{t}
\]
where $b(X_{t})=\sigma(X_{t})^{T}\sigma(X_{t})$ and $V_{t}$ is a
Wiener process.

Step 4. Since the martingale with initial condition $\mathcal{W}(\psi(0)=x)=1$
completely characterizes $\mathcal{W}$, the above result can be extended
to any $\mathbb{P}$ by the \emph{Principle of Accompanying Laws}
and \emph{Donsker's Invariance Principle} \citep[Theorem 3.1.14 and 3.4.20,][]{Stroock2000}
if and only if $\mathbb{P}$ belongs to the family of all tight measures,
$\mathcal{M}(\mathfrak{P}(\mathbb{S}))$. In our setup, $\mathbb{S}$
is a compact metric space so the collection of $\mathbb{P}(\cdot)$
over $\mathcal{S}$ is tight. The Principle of Accompanying Laws says
that if a sequence is in a complete separable space with tight measure, the
law of this sequence will weakly converge. Donsker's Invariance Principle
says that for independent increment processes, the convergent law is the
law of the Wiener process. Therefore, $\mathbb{P}\in\mathcal{M}(\mathfrak{P}(\mathbb{S}))$
and
\[
\left(f(X_{t})-f(X_{0})-\int_{0}^{t}(\mathbb{A}f)dt,\mathcal{PB}_{t},\mathbb{P}\right)
\]
is a martingale.

The IF claim says that a martingale exists for $f(t,X_{t})$ on $(\Omega,\mathcal{F},\mathbb{P})$.
The maximum principle restricts the process to be $\mathcal{PB}_{t}$-adapted,
thus $\mathcal{F}\sim\mathcal{PB}$ and the result holds on $(\mathbb{S},\mathcal{S})$
with the probability space $(\Omega,\mathcal{F},\mathbb{P})$.
\end{proof}

\subsubsection{Proof of Theorem \ref{thm:riskfree}}
\begin{proof}
(i) Part of the proof follows by Propositions 3.13 and 3.15 in \cite{BainCrisan2008}. The boundedness condition
\[
\mathbb{E}\left[\exp\left(\frac{1}{2}\int h(X_{s})^{2}ds\right)\right]<\infty,
\]
 is called Novikov's condition. By this condition, Girsanov's theorem
implies that $Z_{t}$ defined as
\[
\left.\frac{d\tilde{\mathbb{P}}}{d\mathbb{P}}\right|_{\mathcal{F}_{t}}=Z_{t}:=\exp\left(-\int_{0}^{t}h(X_{s})dW_{s}-\frac{1}{2}\int_{0}^{t}h(X_{s})^{2}ds\right)
\]
is an $\mathcal{F}_{t}$-adapted martingale. The Martingale representation
theorem implies that
\[
W_{t}+\left\langle \int_{0}^{t}h(X_{s})dW_{s},W_{t}\right\rangle _{t}=W_{t}+\int_{0}^{t}h(X_{s})ds=Y_{t},
\]
 where $\langle\cdot,\cdot\rangle_{t}$ is the quadratic variation
such that $\langle W_{t},W_{t}\rangle_{t}=t$. Thus, for $d\tilde{\mathbb{P}}=Z_{t}d\mathbb{P}$,
$Y_{t}$ is a Wiener process with respect to $\tilde{\mathbb{P}}$:
\begin{align*}
\mathbb{E} & e^{\left(W_{t}+\int_{0}^{t}h(X_{s})ds\right)}e^{\left(-\int_{0}^{t}h(X_{s})dW_{s}-\frac{1}{2}\int_{0}^{t}h(X_{s})^{2}ds\right)}\\
=\mathbb{E} & e^{\left\{ \int_{0}^{t}(1+h(X_{s}))dW_{s}-\int_{0}^{t}(2h(X_{s})+h^{2}(X_{s}))ds\right\} }\\
=\mathbb{E} & e^{t^{2}/2}\cdot e^{\left\{ \int_{0}^{t}(1+h(X_{s}))dW_{s}-\int_{0}^{t}(1+h(X_{s}))^{2}ds\right\} }=e^{t^{2}/2}.
\end{align*}
 The last line is the result of (\ref{eq:baseMartingale}).

The law of the pair process $(X,Y)$ can be written as
\[
(X_{t},Y_{t})=(X_{t},W_{t})+\left(0,\int_{0}^{t}h(X_{s})ds\right).
\]
Thus, on an arbitrary time interval $[0,t]$, under the $\tilde{\mathbb{P}}$-law,
the law of $(X_{t},W_{t})$ is absolutely continuous with respect
to the law of the pair process $(X_{t},Y_{t})$. For any bounded measurable
function $\varphi$ defined on the product path space of $(X,Y)$,
we have
\[
\tilde{\mathbb{E}}\left[\varphi(X_{t},Y_{t})\right]=\mathbb{E}\left[\varphi(X_{t},Y_{t})Z_{t}\right]=\mathbb{E}\left[\varphi(X_{t},W_{t})\right].
\]
 Therefore, $X$ and $Y$ are independent under $\tilde{\mathbb{P}}$
since $X$ and $W$ are independent.

(ii) Under the probability measure $\tilde{\mathbb{P}}$, the law
of the process $Y$ is completely specified as an $\mathcal{F}_{t}$-adapted
Wiener process with independent increments of $Y$. Hence, the $\sigma$-algebra
is $\mathcal{Y}_{t}^{\dagger}=\sigma(Y_{t+u}-Y_{t})$ for any $u\geq0$.
Note that $\mathcal{Y}_{t}$ and $\mathcal{Y}_{t}^{\dagger}$ are
independent. By the conditional expectation property,
\[
\tilde{\mathbb{E}}\left[\varphi(X_{t})|\mathcal{Y}_{t}\right]=\tilde{\mathbb{E}}\left[\varphi(X_{t})|\sigma(\mathcal{Y}_{t},\mathcal{Y}_{t}^{\dagger})\right].
\]
 Since $\mathcal{Y}_{t}^{\dagger}$ includes all the incremental information
after time $t$,
\[
\sigma(\mathcal{Y}_{t},\mathcal{Y}_{t}^{\dagger})=\mathcal{Y}_{t}\vee\mathcal{Y}_{(t'-t)\in\mathbb{R}}=\mathcal{Y},
\]
 and $\mathcal{Y}$ is a time invariant $\sigma$-algebra.
\end{proof}

\subsubsection{Proof of Theorem \ref{thm:KSP}}
The proof follows the results given in \cite{Rozovsky1991}.
First we give the Zakai equation, and then we show that KSP is a normalized Zakai equation.
\begin{prop}
\label{pro:Zakai} (Zakai Equation) If $\tilde{\mathbb{E}}[\varphi(X_{t})Z_{t}|\mathcal{Y}_{t}]$
is bounded under $\tilde{\mathbb{P}}$, where
\[
Z_{t}=\exp\left(-\int_{0}^{t}h(X_{s})dW_{s}-\frac{1}{2}\int_{0}^{t}h(X_{s})^{2}ds\right),
\]
 then for any $\varphi\in C_{b}(\mathbb{S})$ the process $\rho_{t}(\varphi):=\tilde{\mathbb{E}}[\varphi(X_{t})Z_{t}|\mathcal{Y}_{t}]$
follows
\[
\rho_{t}(\varphi)=\pi_{0}(\varphi)+\int_{0}^{t}\rho_{s}(\mathbb{A}\varphi)ds+\int_{0}^{t}\rho_{s}(\varphi h)dY_{s}
\]
 on $\tilde{\mathbb{P}}$ almost surely.
\end{prop}

\begin{proof}
Note that if $Z_{t}$ is a $\tilde{\mathbb{P}}$-martingale, then
\[
Z_{t}=\exp\left(-\int_{0}^{t}h(X_{s})dY_{s}-\frac{1}{2}\int_{0}^{t}h(X_{s})^{2}ds\right)
\]
 since $Y_{t}$ is a Wiener process under $\tilde{\mathbb{P}}$. By
Girsanov's theorem
\[
Z_{t}=1+\int_{0}^{t}Z_{t}h(X_{s})dY_{t}.
\]
 Because $\rho_{t}(\varphi)$ is bounded, Fubini's theorem and It\^o's
lemma imply
\begin{equation}
d \rho_{t}(\varphi) =  d\tilde{\mathbb{E}}[\varphi(X_{t})Z_{t}|\mathcal{Y}_{t}]=\tilde{\mathbb{E}}[\mathbb{A}\varphi(X_{t})Z_{t}|\mathcal{Y}_{t}]dt+\tilde{\mathbb{E}}[\varphi(X_{t})h(X_{s})Z_{t}|\mathcal{Y}_{t}]dY_{t}.\label{eq:unormalizedZakai}
\end{equation}
 Taking the integral, we have the result.
\end{proof}

We now turn to the proof of Theorem \ref{thm:KSP}.
\begin{proof}
If a new measure is constructed under a Wiener process $Y$, then
$\pi$ has a representation in terms of $\rho$ by Bayes' rule such
that
\begin{equation}
\pi_{t}(\varphi)=\frac{\rho_{t}(\varphi)}{\tilde{\mathbb{E}}[Z_{t}|\mathcal{Y}_{t}]}=\frac{\rho_{t}(\varphi)}{\exp\left(\int\pi_{s}(h)dY_{s}-\frac{1}{2}\int_{0}^{t}[\pi_{s}(h)]^{2}ds\right)}.\label{eq:changemeasure}
\end{equation}
Since $\rho_{t}(\cdot)$ satisfies a linear evolution equation, we
expect this will lead to an evolution equation for $\pi$. From equation
(\ref{eq:changemeasure}), we have
\begin{equation}
d\left(\frac{1}{\tilde{\mathbb{E}}[Z_{t}|\mathcal{Y}_{t}]}\right)=\frac{1}{\tilde{\mathbb{E}}[Z_{t}|\mathcal{Y}_{t}]}\left(\int\pi_{s}(h)dY_{s}-\frac{1}{2}\int_{0}^{t}[\pi_{s}(h)]^{2}ds\right)\label{eq:KSPproof}
\end{equation}
 which is equivalent to
\[
\pi_{t}(\varphi)=\rho_{t}(\varphi)\cdot\frac{1}{\tilde{\mathbb{E}}[Z_{t}|\mathcal{Y}_{t}]}.
\]
 Note that integration by parts implies
\[
\rho_{t}(\varphi)\cdot\frac{1}{\tilde{\mathbb{E}}[Z_{t}|\mathcal{Y}_{t}]}=\int\frac{1}{\tilde{\mathbb{E}}[Z_{t}|\mathcal{Y}_{t}]}d\rho_{t}(\varphi)+\int\rho_{t}(\varphi)d\left(\frac{1}{\tilde{\mathbb{E}}[Z_{t}|\mathcal{Y}_{t}]}\right).
\]
 Substituting equation (\ref{eq:unormalizedZakai}) for $\rho_{t}(\varphi)$ and (\ref{eq:KSPproof}) for $d (1 / \tilde{\mathbb{E}}[Z_{t}|\mathcal{Y}_{t}] )$, the result follows. \end{proof}

\subsection{Proof of Other Results}

\subsubsection{Proof of Corollary \ref{cor:master-eq}}
\begin{proof}
Sketch of the proof. Take the transition probability $\mathbb{Q}_{\tau'}$
and expand it w.r.t. $\tau'$ at zero by Taylor's expansion:
\begin{equation}
\mathbb{Q}_{\tau'}(X_{u}|X_{s})=\delta(X_{u}-X_{s})+\tau'\mathcal{W}(X_{u}|X_{s})+o(\tau'),\label{eq:TaylorMarkov}
\end{equation}
where $\delta(\cdot)$ is the delta function%
\footnote{Loosely speaking, delta function is a smooth indicator function such
that the derivative of $\delta(\cdot)$ exists in the weak sense.
Regardless of technical differences, one can think both of them are
identical.%
}. The function $\mathcal{W}(X_{u}|X_{s})$ is the time derivative
of the transition probability at $\tau=0$, called \emph{transition
probability per unit time}. This expression must satisfy the normalization
property, in other words, the integral over $X_{u}$ must equal one.
For this purpose, the above form can be corrected to:
\[
\mathbb{Q}_{\tau'}(X_{u}|X_{s})=(1-\alpha_{0}\tau')\delta(X_{u}-X_{s})+\tau'\mathcal{W}(X_{u}|X_{s})+o(\tau'),
\]
where $\alpha_{0}(X_{s})=\int\mathcal{W}(dX_{u}|X_{s})$. Substituting
the expansion form into Chapman-Kolmogorov equation
\[
\mathbb{Q}_{\tau+\tau'}(X_{u}|X_{s})=(1-\alpha_{0}\tau')\mathbb{Q}_\tau(X_{u}|X_{s})+\tau'\int\mathcal{W}(X_{u}|X_{t})\mathbb{Q}_{\tau}(dX_t|X_s),
\]
then dividing the
equation by $\tau'$, substituting $\alpha_{0}(X_s)$ and letting $\tau'$ go to zero give us
the following result
\[
\frac{\partial}{\partial\tau}\mathbb{Q}_{\tau}(X_{u}|X_{s})=\int\left\{ \mathcal{W}(X_{u}|X_{t})\mathbb{Q}_{\tau}(dX_{t}|X_{s})-\mathcal{W}(dX_{t}|X_{u})\mathbb{Q}_{\tau}(X_{u}|X_{s})\right\}.
\]
This derivation is described in \citet[Chapter 5,][]{Kampen2007}.
\end{proof}

\subsubsection{Proof of Corollary \ref{cor:DiffusionCoeff}}
\begin{proof}
For the first part, the IF claim says that a martingale exists for
$f(t,X_{t})$ on $(\Omega,\mathcal{F},\mathbb{P})$. The maximum principle
restricts the process to be $\mathcal{PB}_{t}$-adapted, thus $\mathcal{F}\sim\mathcal{PB}$
and the result holds on $(\mathbb{S},\mathcal{S})$ with the probability
space $(\Omega,\mathcal{F},\mathbb{P})$. The second part is a standard
result of diffusion processes.
\end{proof}

{\footnotesize{}

}{\footnotesize \par}

\end{document}